\documentclass[twoside]{article}
\usepackage{latexsym,bm}
\usepackage{mathrsfs,amsmath,amssymb,amsthm}
\usepackage[dvipdfm]{graphicx}
\usepackage[all]{xy}
\usepackage{color}
\xyoption{2cell}
\newcommand{\Ol}{{\mathcal O}}
\newcommand{\f}{\varphi}
\newcommand{\E}{{\mathcal E}}
\newcommand{\V}{{\mathcal V}}
\newcommand{\Un}{{\mathcal U}}
\newcommand{\C}{{\mathcal C}}
\newcommand{\El}{{\mathcal L}}
\newcommand{\li}{{\mathcal L}}
\newcommand{\ac}{{\mathcal H}}
\newcommand{\ol}{\overline {\mathcal L}}
\newcommand{\oc}{{\overline  C}}
\newcommand{\pu}{{\mathbb P^1}}
\newcommand{\proj}{\mathbb P}

\DeclareMathOperator{\loc}{\mathrm{Locus}}
\DeclareMathOperator{\cloc}{\mathrm{ChLocus}}
\DeclareMathOperator{\ch}{\mathrm{Chain}}
\newcommand{\ratcurves}{\textrm{Ratcurves}^n(X)}

\newcommand{\cone}{\textrm{NE}}
\DeclareMathOperator{\cycl}{N_1}
\DeclareMathOperator{\pic}{Pic}
\DeclareMathOperator{\Exc}{Exc}

\newcommand{\kd}{-K_X \cdot}

\newcommand{\W}{{\mathcal{W}}}
\newcommand{\conx}[1]{\cone\,(#1,X)}
\newcommand{\cycx}[1]{\cycl(#1,X)}

\newcommand{\twospan}[2]{\langle #1,#2 \rangle}

\newcommand{\rc}[2]{#1 \xymatrix{\ar@{-->}[r] & }{#2}}

\newtheorem{theorem}{Theorem}[section]
\newtheorem*{theorem*}{Theorem}
\newtheorem{lemma}[theorem]{Lemma}

\newtheorem{proposition}[theorem]{Proposition}
\newtheorem{corollary}[theorem]{Corollary}
\newtheorem{example}[theorem]{Example}
\theoremstyle{definition}
\newtheorem{definition}[theorem]{Definition}
\newtheorem{statement}[theorem]{}
\theoremstyle{remark}
\newtheorem{remark}[theorem]{Remark}

\newtheorem{claim}[theorem]{Claim}

\begin{document}

\title{Rationally cubic connected manifolds I:\\ manifolds covered by lines}

\author{Gianluca \textsc{Occhetta} and Valentina \textsc{Paterno}}

\maketitle

\label{startpage}

\begin{abstract} In this paper we study smooth complex projective polarized varieties $(X,H)$ of dimension $ n \ge 2$ which admit a dominating family $V$ of rational curves of $H$-degree $3$, such that two general points of $X$ may be joined by a curve parametrized by $V$, and such that there is a covering family of rational curves of $H$-degree one.\\
Our main result is that the Picard number of these manifolds is at most three, and that, if equality holds, $(X,H)$ has an adjuction theoretic scroll structure over a smooth variety.
\end{abstract}

\begin{figure}[bp]
\footnotesize{{\it Mathematics Subject Classification}: Primary 14M22; Secondary 14J40, 14E30. \\
{\it Key words and phrases}: Rationally connected manifolds, rational curves}
\end{figure}

\section{Introduction}

At the end of the last century the concepts of uniruled and rationally connected varieties were introduced as suitable higher dimensional analogues of ruled and rational surfaces. Uniruled varieties are algebraic varieties that are covered by rational curves, i.e. varieties that contain a rational curve through a general point. Among uniruled varieties, those that contain a rational curve through two general points are especially important. Varieties satisfying this property are called rationally connected and were introduced by  Campana in \cite{Cam92} and by Koll\'ar, Miyaoka and Mori in \cite{KoMiMo}.\par
\medskip
A natural problem about rationally connected varieties is to characterize them by means of bounding the degree of the rational curves connecting pairs of general points;  Ionescu and Russo have recently studied  \emph{conic-connected} manifolds embedded in projective space, i.e. projective manifolds such that two general points may be joined by a rational curve of degree $2$ with respect to a fixed very ample line bundle. In \cite{IoRu}, they proved that conic-connected manifolds $X\subset\mathbb{P}^N$  are Fano and have Picard number $\rho_X$ less than or equal to $2$ and classified the manifolds with Picard number two. A special case of conic-connected manifolds was previously studied
 by Kachi and Sato in \cite{KS}.\par
\medskip
In this paper we will consider  \emph{rationally cubic connected} manifolds (RCC-manifolds, for short), i.e. smooth complex projective polarized varieties $(X,H)$ of dimension $ n \ge 2$ which are rationally connected by rational curves of degree $3$ with respect to a fixed ample line bundle $H$, or equivalently which admit a dominating family $V$ of rational curves of degree $3$ with respect to $H$ such that two general points of $X$ may be joined by a curve parametrized by $V$. \par
\medskip
Unlike in the conic-connected case, there is no constant bounding the Picard number of RCC-manifolds, as shown in Example (\ref{mainex}); the same example shows also that there are RCC-manifolds which are not Fano manifolds and which do not carry a covering family of lines (i.e. curves of degree one with respect to $H$), this last property holding for all conic-connected manifolds of Picard number greater than one.\par
\medskip
These considerations lead us to divide our analysis of RCC-manifolds in two parts: in the present paper we will deal with the ones which are covered by lines, while the remaining ones, which present very different geometric features will be treated elsewhere \cite{OP2}.\\
Rationally cubic connected manifolds covered by lines present more similarity with conic-connected manifolds; the first one is the presence of a bound on the Picard number, with a description of the border case.  

\begin{theorem}\label{main1} Let $(X,H)$ be RCC with respect to a family $V$, and assume that $(X,H)$
is covered by lines. Then $\rho_X \le 3$, and if equality holds then there exist three families of rational curves of $H$-degree one, $\El^1,\li^2,\li^3$ with $[V]=[\li^1]+[\li^2]+[\li^3]$ ([\, ] denotes the numerical class), such that $X$ is  rc$(\li^1,\li^2,\li^3)$-connected.\end{theorem}

In the case of maximal Picard number we also prove a structure theorem, which shows that
RCC-manifolds covered by lines  always have an adjunction theoretic scroll structure over
a smooth variety:

\begin{theorem}\label{main2}
Let $(X,H)$ be RCC with respect to a family $V$, assume that $(X,H)$
is covered by lines and that $\rho_X = 3$. Then there is a covering family of lines
whose numerical class spans an extremal ray of $\cone(X)$ and the associated extremal contraction $\f:X \to Y$ is an adjunction scroll over a smooth variety $Y$.
\end{theorem}

For conic-connected manifolds a stronger result holds, namely conic-connected
manifolds with maximal Picard number have a classical scroll structure; as Example (\ref{?})
shows this is not true for RCC-manifolds, i.e. there are RCC-manifolds with a scroll structure
which has jumping fibers.\par
\medskip
As for the question if a RCC-manifold covered by lines is a Fano manifold, we are not able to provide an answer. The big difference with the conic-connected case, which makes the problem definitely harder,
is that the structure of the cone of curves of the manifold, which now lives in a three-dimensional vector space is not known: a priori, many different shapes and an unknown number of extremal rays are possible.

\section{Background material}

\subsection{Families of rational curves and of rational $1$-cycles}\par
\medskip

\begin{definition} \label{Rf}
A {\em family of rational curves} $V$ on $X$ is an irreducible component
of the scheme $\ratcurves$ (see \cite[Definition II.2.11]{Kob}).\\
Given a rational curve we will call a {\em family of
deformations} of that curve any irreducible component of  $\ratcurves$
containing the point parametrizing that curve.\\
We define $\loc(V)$ to be the set of points of $X$ through which there is a curve among those
parametrized by $V$; we say that $V$ is a {\em covering family} if ${\loc(V)}=X$ and that $V$ is a
{\em dominating family} if $\overline{\loc(V)}=X$.\\
By abuse of notation, given a line bundle $L \in \pic(X)$, we will denote by $L \cdot V$
the intersection number $L \cdot C$, with $C$ any curve among those
parametrized by $V$.\\
We will say that $V$ is {\em unsplit} if it is proper; clearly, an unsplit dominating family is
covering.\\
We denote by $V_x$ the subscheme of $V$ parametrizing rational curves
passing through a point $x$ and by $\loc(V_x)$ the set of points of $X$
through which there is a curve among those parametrized by $V_x$. If, for a general point $x \in \loc(V)$,
$V_x$ is proper, then we will say that the family is {\em locally unsplit}.
Moreover, we say that $V$ is {\em generically unsplit} if the fiber of the double-evaluation map
$$\begin{array}{rllc}
\Pi:&V&\rightarrow& X\times X\\
&[f]&\mapsto&(f(q),f(p))
\end{array}$$
over the general point of its image has dimension at most $0$.\end{definition}

\begin{definition}
Let $U$ be an open dense subset of $X$ and $\pi\colon U \to Z$ a proper
surjective morphism to a quasi-projective variety;
we say that a family of rational curves $V$ is a {\em horizontal dominating family with respect to} $\pi$
if $\loc(V)$ dominates $Z$ and curves parametrized by $V$ are not contracted by $\pi$.
\end{definition}

\begin{definition}\label{CF}
We define a {\em Chow family of rational 1-cycles} $\W$ to be an irreducible
component of  $\textrm{Chow}(X)$ parametrizing rational and connected 1-cycles.
We define $\loc(\W)$ to be the set of points of $X$ through which there is a cycle among those
parametrized by $\W$; notice that $\loc(\W)$ is a closed subset of $X$ (\cite[Definition II.2.3]{Kob}).
We say that $\W$ is a {\em covering family} if $\loc(\W)=X$.\\
If $V$ is a family of rational curves, the closure of the image of
$V$ in $\textrm{Chow}(X)$, denoted by $\V$, is called the {\em Chow family associated to} $V$.
If $V$ is proper,  i.e. if the family is unsplit, then $V$ corresponds to the normalization
of the associated Chow family $\V$.
\end{definition}

\begin{definition}
Let $V$ be a family of rational curves and let $\V$ be the associated Chow family. We say that
$V$ (and also $\V$) is {\em quasi-unsplit} if every component of any reducible cycle parametrized by 
$\V$ has numerical class proportional to the numerical class of a curve parametrized by $V$.
\end{definition}

\begin{definition}
Let $V^1, \dots, V^k$ be families of rational curves on $X$ and $Y \subset X$.
We define $\loc(V^1)_Y$ to be the set of points $x \in X$ such that there exists
a curve $C$ among those parametrized by $V^1$ with
$C \cap Y \not = \emptyset$ and $x \in C$. We inductively define
$\loc(V^1, \dots, V^k)_Y := \loc(V^k)_{\loc(V^1, \dots,V^{k-1})_Y}$.\\
Notice that, by this definition, we have $\loc(V)_x=\loc(V_x)$.
Analogously we define $\loc(\W^1, \dots, \W^k)_Y$  for Chow families $\W^1, \dots, \W^k$ of rational 1-cycles.
\end{definition}

{\bf Notation}: 
If $\Gamma$ is a $1$-cycle, then we will denote by $[\Gamma]$ its numerical equivalence class
in $\cycl(X)$; if $V$ is a family of rational curves, we will denote by $[V]$ the numerical equivalence class
of any curve among those parametrized by $V$.
A proper family will always be denoted by a calligraphic letter.\\
If $Y \subset X$, we will denote by $\cycx{Y} \subseteq \cycl(X)$ the vector subspace
generated by numerical classes of curves of $X$ contained in $Y$; moreover, we will denote 
by  $\conx{Y} \subseteq \cone(X)$
the subcone generated by numerical classes of curves of $X$ contained in~$Y$. We will
denote by $\langle \dots \rangle$ the linear span.

\begin{definition} We say that $k$ quasi-unsplit families $V^1, \dots, V^k$ are numerically independent
if in $\cycl(X)$ we have $\dim \langle [V^1], \dots, [V^k]\rangle =k$.
\end{definition}

For special families of rational curves  we have useful dimensional estimates. The basic one is the following:

\begin{proposition}{\rm(\cite[Corollary IV.2.6]{Kob}\label{iowifam})}
Let $V$ be a family of rational curves on $X$ and $x \in \loc(V)$ a point such that every component of $V_x$ is proper.
Then
  \begin{itemize}
       \item[(a)] $\dim \loc(V)+\dim \loc(V_x) \ge \dim X  -K_X \cdot V -1$;
       \item[(b)] every irreducible component of $\loc(V_x)$ has dimension $\ge -K_X \cdot V -1$.
    \end{itemize}
\end{proposition}

\begin{remark}\label{iowiug}
If $V$ is a generically unsplit dominating family then, for a general $x \in X$ the inequalities in Proposition (\ref{iowifam}) are equalities, by \cite[Proposition II.3.10]{Kob}.
\end{remark}

Proposition (\ref{iowifam}), in case $V$ is the unsplit family of deformations of a minimal extremal
rational curve, gives the {\em fiber locus inequality}:

\begin{proposition} {\rm(\cite{Io,Wicon})}\label{fiberlocus} Let $\f$ be a Fano-Mori contraction
of $X$ and let $E = \Exc(\f)$ be its exceptional locus;
let $G$ be an irreducible component of a (non trivial) fiber of $\f$. Then
$$\dim E + \dim G \geq \dim X + l -1,$$
where $l =  \min \{ -K_X \cdot C\ |\  C \textrm{~is a rational curve in~} G\}.$
If $\f$ is the contraction of a ray $R$, then $l(R):=l$ is called the {\em length of the ray}.
\end{proposition}

The following generalization of Proposition (\ref{iowifam}) will be often used:

\begin{lemma} {\rm(Cf. \cite[Lemma 5.4]{ACO})} \label{locy}
Let $Y \subset X$ be an irreducible closed subset and  $V^1, \dots, V^k$ numerically independent
unsplit families of rational curves such that\linebreak $\langle [V^1], \dots, [V^k]\rangle \cap \conx{Y}={\underline 0}$.
Then either $\loc(V^1, \ldots,V^k)_Y=\emptyset$ or
      $$\dim \loc(V^1, \ldots, V^k)_Y \ge \dim Y +\sum -K_X  \cdot V^i -k.$$
\end{lemma}

A key fact underlying our strategy to obtain bounds on the Picard number, based on \cite[Proposition II.4.19]{Kob},
is the following:

\begin{lemma}\label{numeq}{\rm(\cite[Lemma 4.1]{ACO})}
Let $Y \subset X$ be a closed subset, $\V$ a Chow fami\-ly of rational $1$-cycles. Then every curve
contained in $\loc(\V)_Y$ is numerically equivalent to a linear combination with rational
coefficients of a curve contained in $Y$ and of irreducible components of cycles parametrized
by $\V$ which meet $Y$.
\end{lemma}

The following Corollary encompasses the most frequent usages of Lemma (\ref{numeq}) in the paper:

\begin{corollary}\label{numcor} Let $V^1$ be a locally unsplit family of rational curves,
and $V^2, \dots, V^k$ unsplit families of rational curves. Then, for a general $x \in \loc(V^1)$,
\begin{enumerate}
\item[(a)] $\cycx{\loc(V^1)_x} = \langle [V^1] \rangle$;
\item[(b)]  $\loc(V^1, \dots, V^k)_x= \emptyset$ or 
$\cycx{\loc(V^1, \dots, V^k)_x} = \langle [V^1], \dots, [V^k] \rangle$.
\end{enumerate}
\end{corollary}

\subsection{Contractions and fibrations}\par
\medskip

\begin{definition}
Let $X$ be a manifold such that $K_X$ is not nef.
By the  Cone Theorem
 the closure of the cone of effective 1-cycles into
the $\mathbb R$-vector space of 1-cycles modulo numerical equivalence,
$\overline {\cone(X)} \subset\cycl(X)$, is polyhedral in the part contained in
the set $\{z\in \cycl(X) :K_X \cdot z<0\}$.
An extremal face  is a face of this polyhedral part, and an extremal face of dimension one
is called an {\em extremal ray}.\\
To an extremal face $\Sigma \subset \cone(X)$ is associated a morphism with connected fibers 
$\f_\Sigma:X \to Z$ onto a normal variety, morphism which contracts the curves 
whose numerical class is in $\Sigma$; $\f_\Sigma$ is called an {\em extremal contraction}
or a {\em Fano-Mori contraction}, while a Cartier divisor $H$ such that 
$H = \f_\Sigma^*A$ for an ample divisor $A$ on $Z$
is called a {\em supporting divisor} of the map $\f_\Sigma$ (or of the face $\Sigma$).
We denote with $Exc(\f_{\Sigma}):=\{x\in X|\dim \f_\Sigma ^{-1}(\f_{\Sigma}(x))>0\}$ the {\em exceptional locus} of $\f_\Sigma$.\\
An extremal contraction associated to an extremal ray is called an 
{\em elementary contraction};
an elementary contraction is said to be of {\em fiber type} if $\dim X>\dim Z$, otherwise the contraction is {\em birational}. Moreover, if the codimention of the exceptional locus of an elementary birational contraction is equal to one, then the contraction is called {\em divisorial}; otherwise it is called {\em small}.
\end{definition}

\begin{definition} An elementary fiber type extremal contraction 
$\f: X \to Z$ is called an {\em adjunction scroll} 
if there exists a $\f$-ample line bundle $H \in \pic(X)$ such that
$K_X+(\dim X-\dim Z+1)H$ 
is a supporting divisor of $\f$.
An elementary fiber type extremal contraction $\f:X \to Z$ onto a smooth
variety $Z$ is called a $\mathbb P${\em-bundle} or a {\em classical scroll}  if there exists a vector bundle $\E$ of rank 
$\dim X-\dim Z+1$  on $Z$ such that 
$X \simeq \proj(\E)$. 
Some special scroll contractions arise from projectivization of
B\v anic\v a sheaves (cfr. \cite{BW}); in parti\-cu\-lar,
if $\f:X \to Z$ is a scroll such that every fiber has dimension $\le \dim X- \dim Z+1$,
then $Z$ is smooth and $X$ is the projectivization of a B\v anic\v a sheaf on $Z$
(cfr. \cite[Proposition 2.5]{BW}).
\end{definition}

If $X$ admits a fiber type extremal contraction, then it is uniruled; for the converse,
we have that a covering family of rational curves determines a rational fibration, defined on an open set of $X$. We recall briefly this construction.

\begin{definition}
Let $Y \subset X$ be a closed subset, and let $\V^1, \dots, \V^k$ Chow families of rational $1$-cycles; 
define $\cloc_m(\V^1, \dots, \V^k)_Y$
to be the set of points $x \in X$ such that there exist cycles $\Gamma_1, \dots, \Gamma_m$
with the following properties:
    \begin{itemize}
       \item $\Gamma_i$ belongs to a family $\V^j$;
       \item $\Gamma_i \cap \Gamma_{i+1} \not = \emptyset$;
       \item $\Gamma_1 \cap Y \not = \emptyset$ and $x \in \Gamma_m$,
    \end{itemize}
 i.e. $\cloc_m(\V^1, \dots, \V^k)_Y$, is the set of points that can be
joined to $Y$ by a connected chain  of at most $m$ cycles belonging to the families $\V^j$.

\end{definition}

Define a relation of {\em rational connectedness with respect to
$\V^1, \dots, \V^k$} on $X$ in the following way: two points $x$ and $y$ of $X$ are in rc$(\V^1, \dots, \V^k)$-relation if
there exists a chain of cycles in $\V^1, \dots, \V^k$ which joins
$x$ and $y$,  i.e. if $y \in \cloc_m(\V^1, \dots, \V^k)_x$ for some $m$.
In particular,  $X$ is {\em rc$(\V^1, \dots, \V^k)$-connected} if for some $m$ we have
$X=\cloc_m(\V^1, \dots, \V^k)_x$.\par
\medskip

The families $\V^1, \dots, \V^k$ define proper prerelations in the sense of \cite [Definition IV.4.6]{Kob};
to the proper prerelation defined by $\V^1, \dots, \V^k$ it is associated a fibration, which we will call the
{\em rc$(\V^1, \dots, \V^k)$-fibration}:

\begin{theorem}{\rm(\cite[IV.4.16]{Kob}, Cf. \cite{Cam81})}\label{rcfib} Let $X$ be a normal and proper variety and $\V$
a proper prerelation; then there exists an open subvariety $X^0 \subset X$ and a proper morphism with
connected fibers $\pi\colon X^0 \to Z^0$ such that
\begin{itemize}
\item the rc$(\V^1, \dots, \V^k)$-relation restricts to an equivalence relation on $X^0$;
\item $\pi^{-1}(z)$ is a rc$(\V^1, \dots, \V^k)$-equivalence class for every $z \in Z^0$;
\item $\forall\, z\in Z^0$ and  $\forall\, x,y \in \pi^{-1}(z)$, $x \in \cloc_m(\V^1, \dots, \V^k)_y$ with
$m \le 2^{\dim X -\dim Z}-1$.
\end{itemize}
\end{theorem}

If $\V$ is a covering Chow family of rational $1$-cycles, associated to a quasi-unsplit dominating family $V$, and $\pi: \rc{X}{Z}$ is the rc$(\V)$-fibration, then by \cite[Proposition 1, (ii)]{BCD} its indeterminacy locus $B$ is the union of all rc$(\V)$-equivalence classes
of dimension greater than $\dim X - \dim Z$.\par
\medskip
Combining Theorem (\ref{rcfib}) with Lemma (\ref{numeq}), we get the following:

\begin{proposition} (Cf. \cite[Corollary 4.4]{ACO})\label{rhobound}
If $X$ is rationally connected with respect to some Chow families of rational $1$-cycles
$\V^1, \dots, \V^k$, then $\cycl(X)$ is gene\-rated by the classes of irreducible components of cycles
in $\V^1, \dots, \V^k$.\\
In particular, if $\V^1, \dots, \V^k$ are quasi-unsplit families, then $\rho_X \le k$ and equality
holds if and only if $\V^1, \dots, \V^k$ are numerically independent.
\end{proposition}

\subsection{Extremality of families of rational curves}\par
\medskip
The key observation for proving the extremality of the numerical class of a family of curves is a
variation of an argument of Mori, contained in  \cite[Proof of Lemma 1.4.5]{BSW}.
We state it as follows:

\begin{lemma} \label{numcon}
Let $Z \subset X$ be a closed subset and let $V$ be a quasi-unsplit family of rational curves. Then, for every integer $m$,
every curve contained in $\cloc_m(\V)_Z$ is numerically equivalent to
a linear combination with rational coefficients
   $$\lambda C_Z + \mu C_V,$$
where $C_Z$ is a curve in $Z$, $C_V$ is a curve among those parametrized by $V$ and $\lambda \ge 0$.
\end{lemma}

We build on Lemma (\ref{numcon}), to analyze particular situations which will appear in the proof
of Theorem (\ref{main2}).

\begin{lemma}\label{face}
Let $\li^1, \li^2$ and $\li^3$ be numerically independent unsplit families on $X$.
Assume that we have $X=\cloc_{m_1}(\El^1,\El^2)_{\cloc_{m_2}(\El_3)_x}$ for some point $x \in X$ and some integers $m_1, m_2$. Then the numerical classes $[\El^1], [\El^2]$ lie in a (two-dimensional) extremal face of $\cone(X)$.
\end{lemma}

\begin{proof} First of all notice that, by Proposition (\ref{rhobound}), we have that $\rho_X=3$.\\
By repeated applications of Lemma (\ref{numcon}), starting with  $Z:= \cloc_{m_2}(\El_3)_x$, the numerical class of every curve in $X$ can be written as $\sum_1^3 a_j[\El^j]$,
with $a_3 \ge 0$.\\
Let $\Pi \subset \cycl(X)$ be the plane defined by $[\El^1]$ and $[\El^2]$ and let $C^1$ and $C^2$ be two curves such that $[C^1] + [C^2] \in \Pi$; write $[C^i]= \sum c^i_j[\El^j]$, with $c^i_3 \ge 0$.\\
 Asking for $[C^1] + [C^2]$ to be in $\Pi$ amounts
to impose $c^1_3 + c^2_3 =0$, hence $c^1_3 = c^2_3 =0$ and both $[C^1]$ and $[C^2]$
belong to $\Pi$. 
\end{proof}

\begin{lemma}\label{face2}
Let $\li^1, \li^2$ and $\li^3$ be numerically independent unsplit families on $X$.
Assume that we have $X=\loc(\El^2,\El^1)_{\cloc_{m}(\El_3)_x}$ for some point $x \in X$ and some positive integer $m$. Then $[\li^1]$ is extremal in $\cone(X)$.
\end{lemma}

\begin{proof} By Lemma (\ref{face}) the numerical classes of $\li^1$ and $\li^2$
lie in an extremal face $\Sigma$. Let $C \subset X$ be a curve whose numerical class is contained in $\Sigma$.\\
Since $X=\loc(\El^1)_Z$ with $Z=\loc(\li^2)_{\cloc_{m}(\El_3)_x}$, by Lemma (\ref{numcon})
there is an effective curve $C_{Z} \subset \loc(\li^2)_{\cloc_{m}(\El_3)_x}$ such that
$$C \equiv \alpha C_{Z} + \beta C_1$$
with $C_1$ parametrized by $\li^1$ and $\alpha \ge 0$.\\
Since  $[C] \in \Sigma$, then also $[C_{Z}] \in \Sigma$; on the other hand, by Lemma (\ref{numcon}) applied to $Z$ we have that $[C_{Z}] \in \langle [\li^2], [\li^3] \rangle$, so  $[C_{Z}]=\lambda[\li^2]$.\\
We have thus shown that every curve whose numerical class belongs to $\Sigma$ is equivalent to
$\alpha'[ \li^2] + \beta [\li^1]$ with $\alpha' \ge 0$.
Let now $B^1$ and $B^2$ be two curves such that $[B^1] + [B^2] \in \mathbb R_+[\li^1]$;
 by the extremality of $\Sigma$, both
$[B^1]$ and $[B^2] $ are contained in $\Sigma$.
Write $[B^i] = \alpha_i'[\li^2] + \beta_i[\li^1]$ with $\alpha_i' \ge 0$. Then it is clear that
$[B^1] + [B^2]  \in \mathbb R_+[\li^1]$ if and only if $\alpha'_i=0$ for $i=1,2$.
\end{proof}


\section{Examples} \label{exa}

\begin{example}\label{mainex}{\rm(}{\bf RCC manifolds with large Picard number}{\rm)}\\ Let $P_1,...,P_k$ be general points of $\mathbb{P}^n$  with $$k\leq \left(\begin {array}{c}
n+3\\
3
\end{array}\right)-(2n+2),$$
and let $\f:X \to \proj^n$ be the blow-up of $\mathbb{P}^n$ at $P_1,...,P_k$.
Denote by $E_i$ with  ${i=1,\ldots,k}$  the exceptional divisors.
Let  $V$ be the family of deformations of the strict transform of a general line in $\mathbb{P}^n$ and define
$H$ to be
$$H:=\varphi^*\mathcal{O}_{\mathbb{P}^n}(3)-\left(\sum_{i=1}^{k}E_i\right).$$
By \cite{Co} the line bundle $H$ is very ample, thus the  pair $(X,H)$ is RCC with respect to $V$ and
$\rho_X=k+1$.\\
Notice that, if $k \ge 2$ and $n >2$ then $X$ is not a Fano manifold. In fact, if we consider the strict transform $\ell$ of the line in $\mathbb{P}^n$ passing through $P_1$ and $P_2$, then, by the canonical bundle formula of the blow-up, $-K_X\cdot \ell \leq 0$.
\end{example}

\begin{example} {\rm(}{\bf Products}{\rm)}\\
Let $Y$ be a conic-connected manifold with Picard number two, and denote by $H_Y$ the hyperplane divisor.\\
Trivial examples can be obtained by taking the product $X:=Y \times \proj^r$, with projections $p_1$ and $p_2$ and setting $H$ to be $p_1^*H_Y \otimes p_2^* \Ol_{\proj^r}(1)$.\end{example}

\begin{example}\label{?} {\rm(}{\bf Adjunction scrolls}{\rm)}\\
Let $Y$ be $\proj^{r_1} \times \proj^{r_2} \times \proj^{r_3}$ with $r_i \ge 2$, let $X$ be a general member
of the linear system $|\Ol(1,1,1)|$ and let $H$ be the restriction to $X$ of $\Ol(1,1,1)$.
Then $(X,H)$ is a RCC-manifold which has three extremal contractions which are adjunction scrolls.
If $r_i < r_j +r_k$ then the contraction onto $\proj^{r_j} \times \proj^{r_k}$ has a $(r_j +r_k -r_i -1)$-dimensional family of jumping fibers, hence it is not a classical scroll.
Notice that the condition is always fulfilled by at least two indexes, and, taking $r_1=r_2=r_3$, it is fullfilled
by all, hence in this case $X$ has no classical scroll contractions.
\end{example}

\begin{example}\label{onediv} {\rm(}{\bf Projective bundles}{\rm)}\\
Let $Y$ be  a conic-connected manifold of Picard number two; such manifolds have two extremal contractions  onto projective spaces, one of which is a classical scroll. We denote this contraction by  $\sigma_1$ and the other contraction by $\sigma_2$.
Denote by $\ac_1$ and $\ac_2$ respectively the line bundles  $\sigma_1^*\Ol_\proj(1)$ and 
$\sigma_2^*\Ol_\proj(1)$.\\
For every integer $r \ge 1$ consider the vector bundles $\E_i:= (\Ol_Y)^{\oplus r} \oplus \ac_i$ on $Y$ and their projectivizations $X_i=\proj(\E_i)$, with natural projections $\pi_i: X_i \to Y$.\\
Let $\xi_i$ be the tautological line bundle of $\E_i$, and set $H:= \xi_i + \pi_i^*\ac_1+\pi_i^*\ac_2$; $H$ is the sum of three nef line bundles which do not all vanish on the same curve, hence it is ample.\\
The restriction of $\E_i$ to conics in $Y$ is $\Ol^{\oplus r} \oplus \Ol(1)$; let
$V$ be the family of  sections over smooth conics in $Y$ corresponding to the surjections $(\E_i)_{|\gamma} \to \Ol_\gamma(1)$.\\ We claim that $(X_i,H)$ is RCC with respect to $V$; first of all it is clear that 
$$\xi_i \cdot V= \pi_i^* \ac_1 \cdot V = \pi_i^* \ac_2 \cdot V =1,$$
hence $H \cdot V=3$. Let now $x$ and $x'$ be general points in $X_i$; let $y$ and $y'$ be the images of these points in $Y$ and let $\gamma$ be a conic in $Y$ passing through $y$ and $y'$. By the generality
of $x$ and $x'$ we can assume that $\gamma$ is smooth. Let $\Gamma$ be the projectivization
of the restriction of $\E_i $ to $\gamma$. The variety $\Gamma$ is the blow-up of $\proj^r$ in
a linear subspace $\Lambda$ of codimension two, and a general curve in $V$ contained in $\Gamma$ is
the strict transform of a line in $\proj^r$ not meeting $\Lambda$. By the generality of $x$ and $x'$ there
is a line in $\proj^r$ not meeting $\Lambda$ whose strict transform contains $x$ and $x'$.\par
\medskip
It's straightforward to check that all the manifolds constructed in this way are Fano manifolds with three elementary contractions;
notice that, depending on the choice of $Y$ and $\ac_i$, the other contractions of $X_i$ can
be of different kind, namely:
$$
\begin{array}{|c|c|c|}
\hline
Y & \;\ac_i\; &\mbox{Contractions}\\\hline
\mbox{$\proj^r \times \proj^s$}& \;1-2 \;&\; \;\mbox{Fiber type - Divisorial}\;\;\\\hline
\mbox{Hyperplane section of $\proj^r \times \proj^s$}& 1-2 & \mbox{Fiber type - Divisorial}\\\hline
\mbox{\; Blow-up of $\proj^n$ along a linear subspace \;}& 1 & \mbox{Fiber type - Small} \\\hline
\mbox{\; Blow-up of $\proj^n$ along a linear subspace \;}& 2 & \mbox{Divisorial - Divisorial}\\\hline
\end{array}
$$
\end{example}

\begin{example}\label{twodiv} {\rm(}{\bf More projective bundles}{\rm)}\\
In Example (\ref{onediv}) we considered bundles on all possible conic-connected manifolds
with Picard number two. Other examples can be constructed taking bundles over the 
product of two projective spaces; this is possible because, in this case, through
a pair of general points, there is not just one conic, but a one-parameter family of conics.\\
Let $Y$ be $\proj^r \times \proj^s$ and let $\ac_1$ and $\ac_2$ be as in Example (\ref{onediv}). For every integer $r \ge 1$ consider the vector bundles $\E_{ij}:= (\Ol_Y)^{\oplus r} \oplus \ac_i \oplus \ac_j$ on $Y$ and their projectivizations $X_{ij}=\proj(\E_{ij})$, with natural projection $\pi_{ij}: X_{ij} \to Y$.\\
Let $\xi_{ij}$ be the tautological line bundle of $\E_{ij}$, and set $H:= \xi_i + \pi_{ij}^*\ac_1+\pi_{ij}^*\ac_2$; $H$ is the sum of three nef line bundles which do not all vanish on the same curve, hence it is ample.\\
The restriction of $\E_{ij}$ to conics in $Y$ is $\Ol^{\oplus r} \oplus \Ol(1) \oplus \Ol(1)$; let
$V$ be the family of  sections over smooth conics in $Y$ corresponding to the surjections $(\E_{ij})_{|\gamma} \to \Ol_\gamma(1)$.\\ We claim that $(X_{ij},H)$ is RCC with respect to $V$; first of all it is clear that 
$$\xi_i \cdot V= \pi_i^* \ac_1 \cdot V = \pi_i^* \ac_2 \cdot V =1,$$
hence $H \cdot V=3$. Let now $x$ and $x'$ be general points in $X_{ij}$; we claim that there is at most a finite number of curves in $V$ passing through $x$ and $x'$. If this were not the case, through $x$ and $x'$
there would be a reducible cycle parametrized by $\V$. By Proposition (\ref{connred}), this cycle is composed of three lines (i.e. curves of $H$-degree one); since there is only one dominating family of lines -  the lines in the fibers of $\pi_{ij}$ - and $x$ and $x'$ are general, this is impossible.\\
For every conic $\gamma$ passing through $y=\pi_{ij}(x)$ and $y'=\pi_{ij}(x')$ we can compute the dimension
of the space of curves parametrized by $V$ contained in $\pi_{ij}^{-1}(\gamma)$: it is the dimension
of $H^0((\E_{ij}^{\vee}(1))_{|\gamma})$ minus one, which is $2r +1$.\\
Since there is a one-parameter family of conics passing through $y$ and $y'$, the dimension
of the space of curves $T \subset V$ parametrizing curves meeting $F_y:=\pi_{ij}^{-1}(y)$ and  $F_{y'}:=\pi_{ij}^{-1}(y')$
is $2r +1 +1=2r +2$. \\
Since $F_y$ and $F_y'$ have both dimension $r+1$ and we have proved above that through two general
points there is at most a finite number  of curves parametrized by $T$, we can conclude that
through two general points in $F_y$ and $F_y'$ there is a curve parametrized by $V$.

\par
\medskip
It's straightforward to check that all the manifolds constructed in this way are Fano manifolds with three elementary contractions;
notice that, depending on the choice of  $\ac_i$ and $\ac_j$, the other contractions of $X_{ij}$ can
be of different kind, namely:
$$
\begin{array}{|c|c|c|}
\hline
\;\;\ac_i \;\;&\; \;\ac_j\;\; &\mbox{Contractions}\\\hline
1 & 1 & \mbox{Fiber type - Small}\;\\\hline
1 & 2 & \;\;\mbox{Divisorial - Divisorial}\;\;\\\hline
\end{array}
$$
\end{example}

\section{Preliminaries}

\begin{definition} Let $(X,H)$ be a polarized manifold of dimension $n$; if there exists a dominating family $V$ of rational curves such that $H \cdot V=3$ and through two general points of $X$ there is a curve
parametrized by $V$  we will say that $X$ is Rationally Cubic Connected - RCC for short - with respect to $V$.
\end{definition}

\begin{remark} Notice that we are asking that a general cubic through two general points
is irreducible. Examples in which $(X,H)$ is connected by reducible cycles of degree three
can be constructed by taking any projective bundle over the projective space  with a section.
\end{remark}

Our assumptions on $V$ can be rephrased by saying that for a general point $x \in X$
the subset $\loc(V)_x$ is dense in $X$; by \cite[Proposition 4.9]{De} a general
curve $f:\mathbb{P}^1\rightarrow X$ parametrized by $V$ is a $1$-free curve, i.e.
$$f^{*}TX\simeq\mathcal{O}_{\mathbb{P}^{1}}(a_{1})\oplus...\oplus\mathcal{O}_{\mathbb{P}^{1}}(a_{n})$$
with $a_{1}\geq a_{2}\geq...\geq a_{n}$ and  $a_1\geq 2$, $a_{n}\geq
1$. This implies that 
\begin{equation} \label{degV}
-K_X \cdot V= -K_X\cdot f_*\mathbb{P}^1=\sum_{1}^{n}a_i\geq n+1.
\end{equation}
Since the locus of the corresponding family
of rational $1$-cycles $\V$ is closed and $\loc(V)_x \subset \loc(\V)_x$, we have that
$\loc(\V)_x=X$ for a general $x \in X$.\\
 By Lemma (\ref{numeq}) it follows that
$\cycl(X)$ is generated by the numerical classes of irreducible components of cycles
parametrized by $\V$ passing through $x$. In particular the Picard number of $X$ is one if and only 
if, for some $x \in X$, the subfamily $V_x$ is quasi-unsplit. More precisely we have the
following:

\begin{proposition}
Let $(X,H)$ be RCC with respect to a family $V$; then 
\begin{enumerate}
\item there exists $x \in X$ such that $V_x$ is proper if and only if $(X,H) \simeq (\proj^n,\Ol_{\proj}(3))$;
\item there exists $x \in X$ such that $V_x$ is quasi-unsplit if and only if $X$ is a Fano manifold of Picard number one and index $r(X) \ge \frac{n+1}{3}$ with fundamental divisor $H$.
\end{enumerate}
\end{proposition}

\begin{proof}
In the first case $X$ is the projective space and $V$ is the family of lines by \cite[Proof of Theorem 1.1]{Kepn}.
In the second case the Picard number of $X$ is one by Lemma (\ref{numeq}), hence
$-K_X \ge \frac{n+1}{3} H$ by taking intersection numbers with $V$. \\
The existence of a reducible cycle in $\V$ provides a curve with intersection number one with $H$.
\end{proof}

\section{RCC-manifolds with plenty of reducible cubics}

The results in the previous section show that, if the Picard number of $X$ is greater than one, through
a general point there is at least one reducible cycle in $\V$ whose components are not
all numerically proportional to $V$.
Since $H \cdot V=3$, a cycle in $\V$ can split into two or three irreducible rational components.
From now on we will call a component of $H$-degree one a \textit{line} and a component of $H$-degree two a \textit{conic}.\\
Families of lines are easier to handle, since they cannot degenerate further, i.e., they are unsplit families; for this reason the first possibility that we consider is the following: through
a general point of $X$ there is a reducible cycle consisting of three lines.

\begin{definition}
We will say that a manifold $(X,H)$ which is RCC with respect to $V$ is {\em covered
by $V$-triplets of lines} if through a general point of $X$ there is a connected
rational 1-cycle $\ell^1 + \ell^2 + \ell^3$ such that $[\ell^1] + [\ell^2] + [\ell^3]=[V]$.
\end{definition}

\subsection{RCC-manifolds covered by triplets of lines}\par
\medskip

We start by considering the following - a priori different - situation:

\begin{definition}
We will say that a manifold $(X,H)$ which is RCC with respect to $V$ is {\em connected
by $V$-triplets of lines} if there exist three families of lines $\El^1,\li^2,\li^3$ with $[V]=[\li^1]+[\li^2]+[\li^3]$ such that $X$ is  rc$(\li^1,\li^2,\li^3)$-connected.
\end{definition}

\begin{proposition}\label{3fam} If $(X,H)$ is connected by $V$-triplets of lines then $\rho_X \le 3$.\\ 
If equality holds
then, up to reordering,  $\li^1$ is a covering family, $\li^2$
is horizontal and dominating w.r. to the rc$(\li^1)$-fibration and $\li^3$
is horizontal and dominating w.r. to the rc$(\li^1,\li^2)$-fibration.
\end{proposition}

\begin{proof} The first assertion follows from Proposition (\ref{rhobound}).\\ 
Assume now that $\rho_X=3$; since $X$ is  rc$(\li^1,\li^2,\li^3)$-connected at least one of the families,
say $\li^1$, is covering.\\ Let $\pi_1\colon \rc{X}{Z^1}$ be the rc$(\li^1)$-fibration; since $\rho_X = 3$, by Proposition (\ref{rhobound}), we have 
$\dim Z^1 >0$. Two general fibers of $\pi^1$ are connected by chains of curves in $\li^2$ and $\li^3$,
so one of the families, say $\li^2$, is horizontal and dominating w.r. to $\pi^1$.\\
Let $\pi_2\colon \rc{X}{Z^2}$ be the rc$(\li^1,\li^2)$-fibration; since $\rho_X = 3$, by Proposition (\ref{rhobound}), we have  $\dim Z^2 >0$. Two general fibers of $\pi^2$ are connected by chains of curves parametrized by $\li^3$,
so  $\li^3$ is horizontal and dominating  w.r. to $\pi^2$.
\end{proof}

We show now that for a manifold $(X,H)$, which is RCC with respect to a family $V$ being  covered by $V$-triplets of lines is indeed equivalent to being connected
by $V$-triplets of lines:

\begin{proposition}\label{3l}
Assume that $(X,H)$ is RCC-connected by a family $V$.
Then $(X,H)$ is covered by $V$-triplets of lines if and only if $(X,H)$ is connected by $V$-triplets of lines.
\end{proposition}

\begin{proof}
Consider the set of  triplets of families of lines whose numerical classes add up to $[V]$: $\mathcal S = \{(\li_i^1, \li^2_i,\li^3_i)\, |\, [\li^1_i] +[\li^2_i]+[\li^3_i] =[V] \}_{i = 1, \dots, k}$.\\
For every $i= 1, \dots, k$
denote by $B_i$ the set of points which are
contained in a connected chain $\ell^1 \cup \ell^2 \cup \ell^3$, with $\ell^j$ is parametrized by 
$\El_i^j$ and $\ell^j \cap \ell^{j+1} \not = \emptyset$ for $j=1,2$.
The set $B_i$ can be written as the union of three closed subset: 
\begin{enumerate}
\item $B_i^1:=\loc(\El_i^2,\El_i^1)_{\loc(\El_i^3)}$,
 \item $B_i^2:= e_2(p_2^{-1}(p_2(e_2^{-1}(\loc(\El_i^1))) \cap p_2(e_2^{-1}(\loc(\El_i^3)))))$,
 \item$B_i^3:=\loc(\El_i^2,\El_i^3)_{\loc(\El_i^1)}$.
 \end{enumerate}
where $e_2$ and $p_2$ are the (proper) morphisms defined on the universal family over $\li^2_i$ appearing in the
fundamental diagram
$$
\xymatrix{\Un^2_i  \ar[r]^(.50){e_2} \ar[d]_{p_2} & X\\
 \li^2_i & & }
$$
Notice that the (closed) set $B_i^j$ is exactly
the set of points on curves parametrized by $\El_i^j$ belonging to the chains.\par
\medskip
If through the general point of $X$ there is a $V$-triplet of lines,
then $X$ is contained in the union of the $B_i^j$; since the $B_i^j$ are a finite number
and each of them is closed there is a pair of indexes $(i_0,j_0)$ such that $X$ is contained
in $B_0:= B_{i_0}^{j_0}$.\\ 
By construction the set $B_i^j$ is contained in $\loc(\El_i^j)$, therefore the family $\El_{i_0}^{j_0}$ is covering.\par
\medskip
To simplify notation we denote from now on by $\El^{1}, \El^{2}$ and $\El^{3}$ the
families corresponding to the index $i_0$. We also assume that $j_0=1$ - we don't lose in generality, even if the sets $B_i^j$ have not the same definition for different $j$'s.  \\
Let us consider the rc$(\li^1)$-fibration $\pi^1: \rc{X}{Z^1}$: if $\dim Z^1=0$ then $X$ is rc$(\li^1)$-connected and the statement follows;
otherwise we claim that either $\El^2$ or $\El^3$ is horizontal and dominating
with respect to $\pi^1$.\\
To prove the claim recall that $X$ is covered by connected cycles $\ell^1 + \ell^2 + \ell^3$,
and observe that these cycles are not contracted by $\pi^1$, otherwise also curves
parametrized by $V$ would be contracted, and $Z^1$ should be a point.
Therefore a general fiber of $\pi^1$ meets a cycle $\ell^2 + \ell^3$ and does not contain it, so 
the claim follows.\\
Assume that the family which is  horizontal and dominating with respect to $\pi^1$ is $\El^2$ and consider
the rc$(\El^1,\El^2)$-fibration $\pi^2: \rc{X}{Z^2}$.  If $\dim Z^2=0$ then  $X$ is rc$(\li^1,\li^2)$-connected and the statement follows;
otherwise we can prove, arguing as above, that $\El^3$ is horizontal and dominating
with respect to $\pi^2$.\\
We can thus consider the rc$(\El^1,\El^2,\El^3)$-fibration $\pi^3: \rc{X}{Z^3}$; this fibration contracts the cycles $\ell^1 + \ell^2 + \ell^3$, hence contracts curves parametrized by $V$:
it follows that $\dim Z^3=0$, i.e., $X$ is rc$(\li^1,\li^2,\li^3)$-connected.\par
\medskip
Now we suppose that $(X,H)$ is connected
by $V$-triplets of lines,  i.e. that there exist three families of lines $\El^1,\li^2,\li^3$ with $[V]=[\li^1]+[\li^2]+[\li^3]$ such that $X$ is  rc$(\li^1,\li^2,\li^3)$-connected; we want to prove that $(X,H)$ is covered by $V$-triplets of lines.\\
If all the families of lines are covering, the statement is clear, so we can assume that $\li^3$ is not. 
Let $\pi^2: \rc{X}{Z^2}$ be the rc$(\El^1,\El^2)$-fibration; by Proposition (\ref{iowifam}) its general fiber $F$ has dimension 
$$\dim F \ge -K_X \cdot \li^1 + \dim \loc(\li^2)_x- 1,$$
 for $x$ general in $\loc(\li^2)$; it follows that, for $y$ general in $\loc(\li^3)$: 
\begin{equation}\label{f1}
\dim \loc(\li^3)_y + \dim \loc(\li^2)_x \le n +K_X \cdot \li^1+1 \le -K_X \cdot (\li^2+\li^3);
\end{equation}
on the other hand, by Proposition (\ref{iowifam}) we have 
\begin{equation}\label{f2}
\dim \loc(\li^3)_y + \dim \loc(\li^2)_x \ge -K_X \cdot (\li^2+\li^3) -2.
\end{equation}
Therefore, recalling that $\li^3$ is not covering we have that either $\dim \loc(\li^3)_y = -K_X \cdot \li^3 +1$ or $-K_X \cdot \li^3$.
In the former case $\loc(\li^3)_y$ dominates $Z$, hence 
 $\loc(\li^3,\li^1,\li^2)_y$ is not empty and, by Lemma (\ref{locy}) its dimension is $n$.\\
Otherwise $\dim \loc(\li^3)_x = -K_X \cdot \li^3$ and, by Proposition (\ref{iowifam}), $\loc(\li^3)$ is a divisor $D_3$; since $\li^3$ is horizontal with respect to $\pi^2$, then
$D_3$ is not trivial on the fibers of $\pi^2$.\\
By formulas (\ref{f1}) and (\ref{f2}) and Proposition (\ref{iowifam})  $\loc(\li^2)$ is either $X$ or a divisor $D_2$ (and in this case
clearly $D_2 \cdot \li^1 >0$). In both cases through every point of $X$ there is a connected cycle consisting of a line in $\li^1$ and a line in $\li^2$. The divisor $D_3$ is positive on this cycle hence we get the required triplet of $V$-lines.
\end{proof}

\subsection{RCC-manifolds connected by reducible cubics}\par
\medskip
The next situation we are going to consider is again related to the presence of many
reducible cycles, i.e., we will consider RCC-manifolds such that through two general
points there is a reducible cycle parametrized by the closure of the connecting family.
It turns out that the only such manifolds which are not covered by $V$-triplets of lines
are products of two projective spaces polarized by $\Ol(1,2)$.

\begin{proposition}\label{connred}
Assume that $(X,H)$ is RCC-connected by a family $V$, that, given two general points ${x,x'\in X}$, there exists a reducible cycle parametrized by $\mathbf{\V}$ passing through $x$ and $x'$ and that $X$ is not covered by $V$-triplets of lines.\\ Then $(X,H) \simeq (\proj^t \times \proj^{n-t}, \Ol(1,2))$.
\end{proposition}

\begin{proof}
In view of Proposition (\ref{3l}) we can assume that through a general point there is no reducible cycle parametrized by $\V$ consisting of three lines, hence we can assume that  through two general points there exists a reducible cycle $\ell +  \gamma$ parametrized by $\mathbf{\V}$  consisting of a line and a conic.\\
Consider the pairs $\{(\El^j,C^j)\}_{j=1 \dots, k}$, where $\El^j$ is a family of lines, $C^j$ is a family
of conics with $[\El^j]+[C^j]=[V]$ and let $\mathcal C^j$ be the Chow family associated to $C^j$, with universal family
$\mathcal U_{\mathcal C^j}$.\\
Define, as in \cite[IV.4]{Kob}, $\ch_1(\El^j) = \Un_{\El^j} \times_{\El^j} \Un_{\El^j}$ and 
$\ch_1(\C^j) = \Un_{\C^j} \times_{\C^j} \Un_{\C^j}$.
Set $Y^1_j := \ch_1(\El^j) \times_X \ch_1(\C^j)$ and  $Y_j^2:= \ch_1(\C^j) \times_X \ch_1(\El^j)) $;
our assumptions can be restated by saying that the natural morphism
$$\text{ev}\colon \bigcup^k_{j=1} \;( Y^1_j \cup Y_j^2 )\longrightarrow X \times X$$
is dominant.
Since, for every $j$, the image of $Y_j=Y^1_j \cup Y^2_j$ in $X \times X$ is
closed, the there exists an index $j_0$ such that
$\text{ev}_{|Y_{j_0}}: Y_{ j_0}  \to X \times X$ is surjective. From now on we consider all objects corresponding to this index $j_0$ and we omit it.\\
Denote by $\text{ev}^1$ and $\text{ev}^2$ the restrictions of $\text{ev}$ to $Y^1$ and $Y^2$. The morphism 
$\text{ev}^1$  is the composition of  $\text{ev}^2$  with the involution exchanging the factors of $X \times X$, hence both $\text{ev}^1$ and $\text{ev}^2$ are surjective.\\
For $(x,x')$ to be in the image of $ \text{ev}^1$ (respectively $\text{ev}^2$) means that there is a cycle $\ell + \gamma$ with $\ell$ and $\gamma$ parametrized by $\El$ and $\C$ such that $x \in \ell$ and $x' \in \gamma$ (respectively $x \in \gamma$ and $x' \in \ell$). So, by the surjectivity of $\text{ev}^1$ and $\text{ev}^2$, for every $x \in X$
$$
X=  \loc(\El,\C)_x = \loc(\C,\El)_x.
$$
It follows that both $\li$ and $\C$ are covering and that, for a general $x \in X$, 
$$\dim \loc(\li)_x + \dim \loc(\C)_x = n. $$
For a general $x \in X$ we have that $C_y$ is proper for any point $y \in \loc(\El_x)$; in fact, if this were not the case, then  through $x$ there would be a reducible cycle with numerical
class $[V]$, consisting of three lines, contradicting our assumptions: in particular $C$ is locally unsplit.\\
Applying twice Lemma (\ref{numcon}) we get that  $\cone(X)= \twospan{[\El]}{[C]}$.
By Remark (\ref{iowiug}) we have 
$$-K_X \cdot (\li + C) = \dim \loc(\li)_x + \dim \loc(\C)_x +2 = n+2;$$
it follows that both the extremal contractions of $X$ are equidimensional.\\
 Let $\f_{\li}:X \to Z$ be the contraction of
the ray $\mathbb R_+[\li]$; since $H \cdot \li =1$ we can apply \cite[Lemma 2.12]{Fuj4} to get that
$\f_{\li}$ gives to $X$ a structure of $\proj$-bundle over $Z$: more precisely $X=\proj(\E:={\f_{\li}}_*H)$.\\
A general fiber $F$ of the contraction $\f_{\C}$ is $\loc(C)_x$ for some $x \in F$; we can apply \cite[Theorem 3.6]{Kesing} to get that $(F,H_{|F})$ is $(\proj^{n-t}, \Ol(1))$; therefore by \cite[Theorem 4.1]{Laz} also $Z$ is a projective space.\\
Let $l$ be any line in $Z$; consider $X_l:= \f_{\li}^{-1}(l)=\proj_l(\E_{|l})$; the image of $X_l$ via $\f_{\C}$ has dimension
smaller than $X_l$; the only vector bundle on $\pu$ such that its projectivization has a map (different by the projection onto $\pu$) to a smaller dimensional variety, is the trivial one (and its twists). Therefore $\E$ is uniform of splitting type $(a,a, \dots a)$, hence $\E$ splits.\\
It follows that $X$ is a product of projective spaces, that $C$ is the family of lines in one of the factors
and that $H = \Ol(1,2)$.
\end{proof}

\begin{corollary}\label{n+2} Let $(X,H)$ be RCC with respect to a family $V$ and assume that  $\rho_X >1$. 
If $V$ is not generically unsplit then either $X$ is  covered by $V$-triplets of lines or $(X,H) \simeq (\proj^t \times \proj^{n-t}, \Ol(1,2))$.
\end{corollary}

\begin{proof}
The assertion follows from Mori Bend and Break Lemma; in fact, if $V$ is not generically unsplit then  through two general points ${x,x'\in X}$, there is a reducible cycle in $\mathbf{\V}$.
\end{proof}

\section{RCC manifolds covered by lines: Picard number}

In this section we are going to prove Theorem (\ref{main1}):

\begin{theorem*} Let $(X,H)$ be RCC with respect to a family $V$, and assume that $(X,H)$
is covered by lines. Then $\rho_X \le 3$, and if equality holds then there exist three families of rational curves of $H$-degree one, $\El^1,\li^2,\li^3$ with $[V]=[\li^1]+[\li^2]+[\li^3]$, such that $X$ is  rc$(\li^1,\li^2,\li^3)$-connected.\end{theorem*}
 
In view of Corollary (\ref{n+2}) we can confine to the following situation:
\begin{statement}\label{quasi}
$(X,H)$ is a RCC-manifold with respect to a generically unsplit family $V$, covered by lines and not
connected by  $V$-triplets of lines.
\end{statement}
We will show that, in this setting, we have $\rho_X \le 2$, so we assume,
by contradiction that $\rho_X \ge 3$.\par
\medskip
Since $V$ is generically unsplit, by  \cite[Corollary IV.2.9]{Kob}, we have that 
\begin{equation}
-K_X \cdot V = n+1.
\end{equation}
Consider the set $\mathcal B'=\{(\El^i, C^i)\}$ of pairs of families $(\El^i, C^i)$ such that  through a general point $x \in X$ there  is a reducible cycle $\ell +  \gamma$, with $\ell$ and $\gamma$ parametrized respectively by $\El^i$ and $C^i$. \\
Let $\mathcal B=\{(\El^i, C^i)\}_{i=1}^k$ be a maximal set of pairs in $\mathcal B'$ 
such that $[V],  [\El^1], \dots  [\El^k]$ are numerically independent; if one of the family of lines in the pairs belonging to $\mathcal B'$ is covering, we choose it to be $\li^1$. Denote by $\Pi_i$ the two-dimensional vector subspace of
$\cycl(X)$ spanned by $[V]$ and $[\li^i]$.
By Lemma (\ref{numeq}) we have 
$$\cycl(X) = \langle [V], [\El^1],[C^1], \dots, [\El^k],[C^k] \rangle= \langle [V], [\El^1], [\El^2], \dots,[\El^k] \rangle,$$
 hence the Picard number of $X$ is $k+1$.

\begin{claim}\label{newtriplets}
Let $(\li,C)$ be a pair in $\mathcal B$. If $C$ is a dominating family then it is locally unsplit.
\end{claim}

\begin{proof} 
Assume by contradiction that $C$ is not locally unsplit. Arguing as in Proposition (\ref{3l}) we can show that there are two families of lines $\li'$ and $\li''$ such that  $[\li']+[\li'']=[C]$, $\li'$ is covering and $\li''$ is horizontal
and dominating with respect to the rc$(\li')$-fibration. \\
Since through a general point there is a reducible cycle
$\gamma + \ell$, with $\gamma$ and $\ell$ parametrized by $C$ and $\li$, respectively, then either
curves of $\li$ are contracted by the rc$(\li',\li'')$-fibration or $\li$ is horizontal and dominating with respect
to this fibration.\\
In both cases the rc$(\li',\li'',\li)$-fibration contracts both curves parametrized by $C$ and curves  parametrized by $\li$, hence also curves parametrized by $V$ are contracted and $X$ is connected by $V$-triplets of lines, a contradiction.
\end{proof}

\medskip
\textbf{Case 1:} \,$\li^1$ is not a covering family.\par
\medskip
Denote by $ \li$ the covering family of lines.\\
Since no family of lines in $\mathcal B$ is covering, then the families of conics are dominating. 
Moreover they are locally unsplit, in view of Claim (\ref{newtriplets}).\\
For every $i= 1, \dots, k$ denote by $E_i$ the set $\loc(C^i, \El^i)_x$;
by Lemma (\ref{locy}) it has dimension $\dim E_i \ge n-1$; since 
 $E_i \subset \loc(\El^i)$, the inclusion is an equality and $E_i$ is an irreducible divisor.
 Moreover, by Corollary (\ref{numcor}) we have    $\cycx{E_i} =\langle [C^i], [\El^i] \rangle$.\\
We can  assume that $\li$ is not numerically proportional to $V$, otherwise the
rc$(\li)$-fibration would take $X$ to a point, and $\rho_X=1$ by Proposition (\ref{rhobound}).\\
This implies that there is at least a divisor, say $E_1$,
which is not trivial (hence positive) on $\li$; therefore the family $\li^1$ is horizontal and dominating with respect to the rc$(\li)$-fibration. We can assume that $[\li] \not \in \Pi_1$, otherwise the rc$(\li,\li^1)$-fibration $\pi: \rc{X}{Z}$ would go to a point and $\rho_X=2$.\\
Let $F$ be a fiber of $\pi$ and $x \in F \cap \loc(\li^1)$; then, by Proposition (\ref{iowifam}),
$$\dim F \ge \loc( \El^1,\El)_x \ge \dim \loc(\El^1)_x + 1 \ge \kd  \El^1 +1,$$
hence, recalling that $\kd (\li^1 +C^1)= n+1$,
 $$\dim Z \le n + K_X \cdot \li^1 -1 = (n-1) -(n+1) \kd C^1 = \kd C^1-2.$$
On the other hand, since $[\li] \not \in \Pi_1$ curves of $C^1$ are not contracted by $\pi$ and, 
by Claim (\ref{newtriplets}) $C^1$ is locally unsplit, we have,
by Proposition (\ref{iowifam}) $\dim Z \ge \dim \loc(C^1)_x \ge -K_X \cdot C^1 -1$, a contradiction. \par

\medskip
\textbf{Case 2} \quad $\li^1$ is a covering family.\par
\medskip
We will denote from now on the pair $(\li^1,C^1)$ by $(\El,C)$. If $C$ is quasi-unsplit, then
the rc$(\El, \C)$-fibration (which contracts the curves parametrized by $V$) goes to a point and $\rho_X=2$ by Proposition (\ref{rhobound}).\\
Therefore we can assume, from now on that
 $C$ is not quasi-unsplit. Let $x \in X$ be general; then $C_y$ is proper for any point $y \in \loc(\El_x)$; in fact, if this were not the case, then  through $x$ there would be a reducible cycle with numerical
class $[V]$, consisting of three lines, contradicting our assumptions.\\
By this property, we can apply  Corollary (\ref{numcor}) even if $C$ is not unsplit, to get that $\cycx{\loc(\El,\C)_x} = \langle [\El],[C] \rangle$ and Lemma (\ref{locy}) to get $\dim \loc(\li,\C)_x \ge n-1$; if the inequality is strict, then we get the contradiction $\rho_X=2$ by Lemma (\ref{numeq}).
(Notice that this is always the case if $n=2$, so from now on we can assume $n \ge 3$).\\
If equality holds, then an irreducible component of $ \loc(\El,\C)_x$ is a divisor, that we will call $D_x$.\\
If the intersection number $D_x \cdot \El$, which is nonnegative since $\El$ is a covering family, is positive,  we  have
 $X=\loc(\El)_{D_x}$ and again the contradiciton $\rho_X=2$ by Lemma (\ref{numeq}).\\
If else  $D_x \cdot \El=0$, then  every curve of $\El$ which meets $D_x$ is contained in it; 
in particular this implies that $x \in D_x$; this has two important consequences: the first one is that  $D_x \cdot V >0$; in fact being general, $x$ can be joined to another general point $x' \not \in D_x$ by a curve parametrized by $V$. The second one is that, since $x \in D_x \subset \loc(C)$ and $x$ is general, then $C$ is a dominating family, and so
it is locally unsplit by Claim (\ref{newtriplets}).\par
\medskip
Let  $(\ol,\oc) \in \mathcal B$ be a pair different from $(\El,C)$. 
If $\ol$ is not covering
then $\oc$ is dominating (and locally unsplit, by Claim (\ref{newtriplets})).\\ Then, since $x$ is general, $D_x$ meets a general curve of $\oc$; since $[\oc] \not \in \cycx{D_x}$ we have $D_x \cdot \overline C > 0$ and hence, by the same reason, $\dim \loc(\overline C_x)=1$,
forcing $-K_X \cdot \overline C=2$.\\
Recalling that $\kd (\ol + \oc) = \kd V = n+1$, we have $\kd \overline \El=n-1$, hence, by Lemma (\ref{locy}) we have $\dim  \loc(\overline \El, \El)_x = n$, and $\rho_X=2$ by Corollary (\ref{numcor}), a contradiction.\par
 \medskip
If also $\ol$ is covering we can repeat all the above arguments for the pair $(\ol, \oc)$. For a general $x$ we thus
have two divisors $D_x$ and $\overline{D}_x$, which clearly have non empty intersection.
In particular $D_x$ meets both $\loc(\ol)_x$ and $\loc(\oc)_x$. Since $D_x$ cannot contain
curves proportional either to $[\ol]$ or to $[\oc]$ we have $\dim \loc(\ol)_x=\dim \loc(\oc)_x=1$,
hence 
$$n+1 = \kd (\ol + \oc) \le \dim \loc(\ol)_x +1 +\dim \loc(\oc)_x +1 = 4.$$

So we are left with the case $n=3$. We can write  $X= \loc(\li)_{\overline{D}_x}$ for a general $x$, hence $\rho_X = 3$ by Corollary (\ref{numcor}); moreover, by Lemma (\ref{face2}) we have that
 $[\li]$ generates an extremal ray of $\cone(X)$ and that  $[\oc]$ belongs to a two dimensional
face $\Sigma$ of $\cone(X)$ containing $[\li]$.\\ 
Let  $\ell'\cup \ell''$ be a reducible cycle  parametrized by $\overline \C$, whose components are not numerically proportional; notice that $[\ell']$ and $[\ell'']$ belong to $\Sigma$.\\
The divisor $K_X +H$ is trivial on $\oc$, and is negative on $\li$, therefore, up to exchange the cycles, it is negative
on $\ell'$ and positive on $\ell''$. In particular we have $-K_X \cdot \ell' \ge 2$.
Let $\li'$ (resp. $\li''$) be a family of deformations of $\ell'$ (resp. $\ell''$); if $\li'$ is not covering then $X=\loc(\ol)_{\loc(\li')_x}$ and
$\rho_X=2$, a contradiction. 
If else $\li'$ is covering, then we can take the rc$(\ol,\li',\li'')$-fibration, which goes to a point, hence $X$ is connected by $V$-triplets of lines, against the assumptions. \qed

\section{RCC manifolds covered by lines: scroll structure}

In this section we are going to prove Theorem (\ref{main2}):

\begin{theorem*}
Let $(X,H)$ be RCC with respect to a family $V$, assume that $(X,H)$
is covered by lines and that $\rho_X = 3$. Then there is a covering family of lines
whose numerical class spans an extremal ray of $\cone(X)$ and the associated extremal contraction $\f:X \to Y$ is an adjunction scroll over a smooth variety $Y$.
\end{theorem*}

\begin{proof} By Theorem (\ref{main1}) and Proposition (\ref{3fam}) we know that
there exist three families of lines $\El^1, \El^2, \El^3$ such that $[\li^1]+[\li^2]+[\li^3]=[V]$.\\
Moreover $\El^1$ is covering, $\El^2$ is horizontal and dominating with respect to the
rc$(\El^1)$-fibration  $\pi^1: \rc{X}{Z^1}$, and $\El^3$ is horizontal and dominating with respect to the rc$(\El^1,\El^2)$-fibration $\pi^2: \rc{X}{Z^2}$.\\
We will first show that  among the families $\li^i$ which are covering there is (at least) one whose numerical class generates an extremal ray of $\cone(X)$. 
To this end, we divide the proof into cases, according to the number of families among the
$\El^i$ which are covering; notice that, as shown by the examples in Section \ref{exa} all
cases do really occur.\par
\bigskip
{\bf Case 1:} The families $\El^i$, $i=1, \dots, 3$ are all covering.\par
\medskip
Assume that $[\El^3]$ does not span an extremal ray; then, by \cite[Proposition 1, (ii)]{BCD} there is a rc$(\El^3)$-equivalence class of dimension greater than the general one, hence an irreducible component $G$ of this class of dimension $\dim G \ge -K_X \cdot \El^3$.\\
Consider $\loc(\El^1,\El^2)_{G}$; by Lemma (\ref{locy}) its dimension is at least $n-1$; if this dimension
is $n$ then  $[\El^1], [\El^2]$ lie in a two-dimensional
extremal face of $\cone(X)$ by Lemma (\ref{face}).\\
We can draw the same conclusion if an irreducible component of $\loc(\El^1,\El^2)_{G}$ is a divisor $D$.
Infact, if $D$ is positive either on $\El^1$ or on $\El^2$ we have $X=\cloc_{m_1}(\El^1,\El^2)_G$ and we apply again Lemma (\ref{face}). If else $D \cdot \El^1= D \cdot \El^2 =0$, recalling that the numerical class in $X$ of every curve in $D$ can be written as $\sum a_i [\El^i]$ with $a_3 \ge 0$ by Lemma (\ref{numcon}), we get that $D_{|D}$ is nef,
hence $D$ is nef and is a supporting divisor of a face which contains $[\El^1]$ and $ [\El^2]$.\par
We can repeat the same argument starting from another family, say $\El^2$; therefore we prove that,
if neither $[\El^3]$ nor $[\El^2]$ spans an extremal ray, then $[\El^1]$ belongs to two different extremal
faces of $\cone(X)$, hence to an extremal ray.\par
\bigskip

{\bf Case 2:} Two families among the $\El^i$, $i=1, \dots, 3$ are covering.\par
\medskip
If the second covering family is $\El^3$, then it is horizontal and dominating with respect to $\pi^1$; 
moreover, since $X$ is rc$(\El^1,\El^2,\El^3)$-connected, $\El^2$  will be horizontal and
dominating with respect to the rc$(\El^1,\El^3)$-fibration, so, withouth loss of generality
we can assume that $\El^2$ is covering and $\El^3$ is not.\par
\medskip
Assume that $\text{codim} \loc(\El^3) \ge 2$; by Proposition (\ref{iowifam}) we have, for any
$x \in \loc(\El^3)$,  that $\dim \loc(\El^3)_x \ge -K_X \cdot \El^3 +1$.
By Lemma (\ref{locy}), for such an $x$  we have $X= \loc(\El^3,\El^2,\El^1)_x$, hence
$[\li^1]$ is extremal by Lemma (\ref{face2}).\par
\medskip
Assume now that, for a general  $x \in \loc(\El^3)$ we have  $\dim \loc(\El^3)_x = -K_X \cdot \El^3$;
by Proposition (\ref{iowifam}) then $\loc(\El^3)$ is a divisor, that we denote by $D^3$; since
$\li^3$ is horizontal and dominating with respect to $\pi^2$, then, for some $i=1,2$ we have
$D^3 \cdot \li^i >0$.\\
Consider $\loc(\li^2,\li^1)_{\loc(\li^3)_x}$; if it equals $X$, then $[\li^1]$ is extremal by
Lemma (\ref{face2}).
If else   $\loc(\li^2,\li^1)_{\loc(\li^3)_x}$ has dimension $n-1$, we take $D$ to be an irreducible component.
If  $D$ is positive on $\li^1$ or $\li^2$ then $X=\cloc_{m_1}(\El^1,\El^2)_{\loc(\li^3)_x}$ and $[\li^1]$ and $[\li^2]$ are in an extremal face by Lemma (\ref{face}).\\
So assume that $D \cdot \li^1= D \cdot \li^2 = 0$; as in Case 1 we can prove that $D$ is nef.\\ 
Since $D$ is trivial on $\li^1$ and $\li^2$ then $[\li^1]$ and $[\li^2]$ are contained in an extremal  face $\Sigma$.\par
If neither $[\li^1]$ nor $[\li^2]$ is extremal, then, being $-K_X$ positive on $\li^1$ and $\li^2$, there is an extremal ray $R$ in $\Sigma$.
Without loss of generality we can assume that $[\li^1]$ is in the interior of the cone spanned by  $[\li^2]$ and $R$.
Let $B$ be the indeterminacy locus of $\pi^1$, let $D_Z$ be a very ample divisor on $\pi^1(X \setminus B)$ and let $ \widehat D_1:=\overline{(\pi^1)^{-1}D_Z}$. The divisor   $\widehat D_1$ is trivial on $\li^1$ and positive on $\li^2$, since this family is covering, hence $\widehat D_1$ is negative on $R$, so the exceptional locus of $R$ is contained in the indeterminacy locus
 of $\pi^1$ and so has codimension at least two in $X$.\\
Let $F$ be a fiber of the contraction associated to $R$; by Proposition (\ref{fiberlocus}) $F$
has dimension $\dim F \ge 2$.  Then $\dim \loc(\El^1,\El^2)_{F} \ge
 -K_X \cdot (\El^1+\El^2)$ by Lemma (\ref{locy}). \\
Since $D^3 \cdot \El^i>0$ for some $i$
the intersection $\loc(\El^3)_x \cap  \loc(\El^1,\El^2)_F$ is not empty for some $x$, therefore
$$\dim \loc(\El^3)_x \cap  \loc(\El^1,\El^2)_F \ge \kd (\li^1+\li^2+\li^3) - n \ge 1,$$
a contradiction, since the numerical class of every curve in $\loc(\El^1,\El^2)_F$ belongs to $\Sigma$
and $[\li^3]$ does not.\par
\bigskip
{\bf Case 3:} Only $\El^1$ is covering.\par
\medskip

Let $F$ be a general fiber of $\pi^2$; it contains $\loc(\El^2,\El^1)_x$ for some $x\in \loc(\El^2)\cap F$, hence, by Proposition (\ref{iowifam}), it has dimension
\begin{equation}\label{1}
 \dim F \ge -K_X \cdot (\El_1 +\El_2)-1 .
\end{equation}
It follows that $\loc(\li^3)_x$ - which has dimension $\loc(\li^3)_x \ge \kd \li^3$ by the same proposition - dominates the target, hence 
\begin{equation}\label{2}
\dim Z^2=\dim X - \dim F \ge \dim \loc(\li^3)_x.
\end{equation}
Combining (\ref{1}) and (\ref{2}) we get that $\dim \loc(\li^3)_x = -K_X \cdot \li^3 = \dim Z^2$;
by the first equality and  Proposition (\ref{iowifam}) we get that  $\dim \loc(\li^3) = n-1$, while from the second
we infer that, for some $m$,  $X= \cloc_m(\El^1,\El^2)_{\loc(\El^3)_x}$ and
we apply Lemma (\ref{face}) to get that
$[\El^1]$ and $[\El^2]$ are contained in an extremal face $\sigma$.\\
Let $D^3$ be an irreducible component of $\loc(\li^3)$. If $D^3 \cdot \li^1 >0$ then we can exchange the role of $\El^2$ and $\El^3$ and obtain that
$[\El^1]$ and $[\El^3]$ belong to a face $\Sigma'$. Therefore, belonging to two different
extremal faces, $[\li^1]$ belongs to an extremal ray.\\
If else $D_3 \cdot \li^1=0$ then $D^3 \cdot \li^2 >0$, since $\li^3$ is horizontal and dominating with respect to $\pi^2$. Therefore, if $[\li^1]$ were not extremal in $\Sigma$, then there
would be a curve $C$, with $[C] \in \Sigma$ such that $D_3 \cdot C <0$.
Counting dimensions we can write $D^3=\loc(\El^2,\El^1,\El^3)_x=\loc(\li^1,\li^3)_{\loc (\El^2)_x}$, hence the numerical class of every
curve in $D^3$ can be written as $\alpha [\El^1] + \beta [\El^2] + \gamma [\El^3]$ with $\beta \ge 0$.
In particular every curve in $D^3$ whose class is in $\Sigma$  is numerically equivalent to $\alpha [\li^1] + \beta [\li^2]$ with $\beta \ge 0$, so $D^3$ is nef on $\Sigma$,
a contradiction.\par
\medskip
We have thus shown the first part of the statement; we can assume that the covering family which spans an extremal
ray $R$ of $\cone(X)$ is $\li^1$.
We will now show that there is an extremal contraction of $X$ which
is a special B\v{a}nic\v{a} scroll.\\
Let $\f:X \to Y$ be the contraction associated to $R$. If the general fiber of $\f$ has dimension
$-K_X \cdot \li^1-1$ and any fiber has dimension $\le -K_X \cdot \li^1$, then $\f$ is a special B\v{a}nic\v{a} scroll.\\
We will prove that, if this is not the case, then there is another contraction of $X$
which is a projective bundle.\par
\medskip
Assume first $\f$ has a fiber $F$ of dimension $\ge -K_X \cdot \li^1 +1$. Pick $x$ such that
 $\loc(\li^3,\li^2)_x$ meets $F$; for such a point we have 
 $$\dim \loc(\li^3,\li^2)_x + \dim F \le n,$$ 
 hence $$ \dim \loc(\li^3,\li^2)_x\le -K_X \cdot (\li^2+\li^3) -2,$$
from which we get that $\dim \loc(\li^3)_x = -K_X \cdot \li^3 -1$, and $\li^3$ is covering by Proposition (\ref{iowifam}).
Being $\li^3$ covering, we can swap the roles of $\li^2$ and $\li^3$ and, by the same argument,
we get that also $\li^2$ is covering.\\
It follows that both $\loc(\li^2,\li^3)_F$ and $\loc(\li^3,\li^2)_F$ are nonempty, hence, by Lemma (\ref{locy}), 
we have  $X=\loc(\li^2,\li^3)_F = \loc(\li^3,\li^2)_F$; by Lemma (\ref{face2}) both $[\li^2]$ and $[\li^3]$ span
an extremal ray.\\
Let $\psi:X \to Y'$ be the contraction of $\mathbb R_+[\li^2]$; we have 
$$\dim Y' \le n +K_X \cdot \li^2 +1 \le -K_X \cdot (\li^1+\li^3);$$
 on the other hand,
since no curve in $\loc(\li^3)_F$ is contracted by $\psi$ we have  
$$\dim Y' \ge \dim \loc(\li^3)_F \ge \dim F -K_X \cdot \li^3 -1 \ge -K_X \cdot (\li^1+\li^3).$$ 
It follows that all inequalities are equalities; in particular a general fiber of $\psi$ has dimension $-K_X \cdot \li^2-1$ and
$\loc(\li^3)_F$ meets every fiber, hence $\psi$  is 
equidimensional. Recalling that $H \cdot \li^2=1$ we get that $\psi$ is a projective bundle by \cite[Lemma 2.12]{Fuj4}.\par
\medskip
Assume now that every fiber of $\f$ has dimension $-K_X \cdot \li^1$.\\
If $\li^2$ is not a covering family then fibers of $\pi^2$ have dimension $\ge -K_X \cdot (\li^1+\li^2)$, therefore 
$\li^3$ is covering. Therefore at least one among $\li^2$ and $\li^3$ is covering. Assume that it is $\li^2$.\\
If $[\El^2]$ does not span an extremal ray then, by \cite[Proposition 1, (ii)]{BCD}, there is a rc$(\El^2)$-equivalence class of dimension greater than the general one, hence an irreducible component $G$ of this class of dimension $\dim G \ge -K_X \cdot \El^2$.\\
Let  $\widetilde G= \f^{-1}(\f(G))$; it has dimension $\dim \widetilde G \ge -K_X \cdot (\li^1+\li^2)$.
For some $x$ we have $\loc(\li^3)_x \cap \widetilde G \not = \emptyset$, which yields
$$\dim \loc(\li^3)_x \le n - \dim \widetilde G \le -K_X \cdot \li^3 -1,$$
and $\li^3$ is a covering family.\\
By Lemma (\ref{locy}) we have $X=\loc(\li^3)_{\widetilde G}$, hence 
 $\widetilde G$ meets all the rc$(\li^3)$ classes, which are thus equidimensional, and $\li^3$ is extremal.
As above we can show that the associated contraction is a projective bundle.
\end{proof}

\bigskip
\noindent
\small{{\bf Acknowledgements}. 
We would like to thank Francesco Russo and Massimiliano Mella for many useful comments and suggestions.\par
\medskip
\noindent
\small{{\bf Note}. 
This work grew out of a part of the second author's Ph.D. thesis \cite{P} at the
Department of Mathematics of the University of Trento.\par

{\small \textsc{Gianluca Occhetta} \\
Dipartimento di Matematica \\
Universit\`a degli Studi di Trento \\
Via Sommarive 14\\
I-38050 Povo (TN), Italy\\
e-mail: gianluca.occhetta@unitn.it}\par
\medskip

{\small \textsc{Valentina Paterno} \\
Dipartimento di Matematica \\
Universit\`a degli Studi di Trento \\
Via Sommarive 14\\
I-38050 Povo (TN), Italy\\
e-mail: paterno@science.unitn.it}
\label{finishpage}

\end{document}